\documentclass[a4paper]{amsart}
\usepackage{amssymb}

\input xy
\xyoption{all}

\newcommand{\Comp}{\mathsf{Comp}}

\newcommand{\pr}{\mathrm{pr}}

\newcommand{\A}{\mathcal A}

\newcommand{\C}{\mathcal C}

\newcommand{\R}{\mathbb R}
\newcommand{\B}{\mathbb B}

\newtheorem{theorem}{Theorem}

\newtheorem{lemma}{Lemma}
\newtheorem{remark}{Remark}

\title{Equilibrium under uncertainty with Sugeno payoff}
\author{Taras Radul}

\begin{document}
\maketitle

Institute of Mathematics, Casimirus the Great University, Bydgoszcz, Poland;
\newline
Department of Mechanics and Mathematics, Lviv National University,
Universytetska st.,1, 79000 Lviv, Ukraine.
\newline
e-mail: tarasradul\@ yahoo.co.uk


\begin{abstract} This paper studies n-player games where players' beliefs about their opponents'
behaviour are capacities. The concept of an equilibrium
under uncertainty was introduced J.Dow and S.Werlang (J Econ. Theory {\bf 64} (1994) 205--224))  for two players and was extended  to n-player
games by  J.Eichberger and D.Kelsey (Games Econ. Behav. {\bf 30} (2000) 183--215).
  Expected utility was expressed by Choquet
integral.  We consider the concept of an equilibrium
under uncertainty in this paper but with expected utility  expressed by Sugeno
integral.  Existence of such an equilibrium is demonstrated using some abstract non-linear convexity on the space of capacities.
\end{abstract}

\section{Introduction}

The classical Nash equilibrium theory is based on fixed point theory and was developed in frames of linear convexity. The mixed strategies of a player are probability (additive) measures on a set of pure strategies. But an interest to Nash equilibria in more general frames is rapidly growing in last decades. There are also  results about  Nash equilibrium  for non-linear convexities. For instance, Briec and Horvath proved in  \cite{Ch} existence of Nash equilibrium point for $B$-convexity and MaxPlus convexity. Let us remark that MaxPlus convexity is
related to idempotent (Maslov) measures in the same sense as linear convexity is related to probability measures.

We can use additive measures only when we know precisely probabilities of all events considered in a game. However it is not a case
in many modern economic models. The decision theory under uncertainty considers a model when probabilities of states are either not known or imprecisely specified. Gilboa \cite{Gil} and Schmeidler  \cite{Sch} axiomatized  expectations expressed by Choquet
integrals attached to non-additive measures called capacities, as a formal approach to decision-making under uncertainty. Dow and Werlang \cite{DW} used this approach for two players game where belief of each player about a choice of the strategy by the other player is a capacity. They introduced some equilibrium notion for such games and proved its existence.  This result was extended onto games with arbitrary finite number of players \cite{EK}.

Kozhan and Zaricznyi introduced in \cite{KZ} a formal mathematical concept of Nash equilibrium of a game where players are allowed to form non-additive beliefs about opponent's decision but also to play their mixed non-additive strategies. Such game is called by authors game in capacities. The expected payoff function was there defined using a Choquet integral.  Kozhan and Zaricznyi proved existence theorem using a linear convexity on the space of capacities which is preserved by Choquet integral.

An alternative to so-called Choquet expected utility model is the qualitative decision theory.   The corresponding expected utility is expressed by Sugeno integral. See for example papers \cite{DP}, \cite{DP1}, \cite{CH1}, \cite{CH} and others.  Sugeno integral chooses a median value of utilities which is qualitative counterpart of the averaging operation by Choquet integral. It was  introduced in \cite{RN} the general mathematical concept of Nash equilibrium of a game in capacities with expected payoff function defined by Sugeno integral. To prove existence theorem for this  case, it was considered some non-linear convexity on the space of capacities generated by capacity monad structure.

It was noticed in \cite{KZ} that "there is no direct interpretation of the game in non-additive mixed strategies". So, formal mathematical concept of Nash equilibrium for capacities considered in \cite{KZ} and \cite{RN} has rather theoretical character.
We consider in this paper the equilibrium notion from \cite{DW} and \cite{EK} for a game  with expected payoff function defined by Sugeno integral. We prove existence of such equilibrium using above mentioned convexity on the space of capacities.

\section{Games with non-additive beliefs} By $\Comp$ we denote the category of compact Hausdorff
spaces (compacta) and continuous maps. For each compactum $X$ we denote by $C(X)$ the Banach space of all
continuous functions on $X$ with the usual $\sup$-norm. In what follows, all
spaces and maps are assumed to be in $\Comp$ except for $\R$ and
maps in sets $C(X)$ with $X$ compact Hausdorff.

We need the definition of capacity on a compactum $X$. We follow a terminology of \cite{NZ}.
A function $\nu$ which assign each closed subset $A$ of $X$ a real number $\nu(A)\in [0,1]$ is called an {\it upper-semicontinuous capacity} on $X$ if the three following properties hold for each closed subsets $F$ and $G$ of $X$:

1. $\nu(X)=1$, $\nu(\emptyset)=0$,

2. if $F\subset G$, then $\nu(F)\le \nu(G)$,

3. if $\nu(F)<a$, then there exists an open set $O\supset F$ such that $\nu(B)<a$ for each compactum $B\subset O$.

We extend a capacity $\nu$ to all open subsets $U\subset X$ by the formula $\nu(U)=\sup\{\nu(K)\mid K$ is a closed subset of $X$ such that $K\subset U\}$.

It was proved in \cite{NZ} that the space $MX$ of all upper-semicontinuous  capacities on a compactum $X$ is a compactum as well, if a topology on $MX$ is defined by a subbase that consists of all sets of the form $O_-(F,a)=\{c\in MX\mid c(F)<a\}$, where $F$ is a closed subset of $X$, $a\in [0,1]$, and $O_+(U,a)=\{c\in MX\mid c(U)>a\}$, where $U$ is an open subset of $X$, $a\in [0,1]$. Since all capacities we consider here are upper-semicontinuous, in the following we call elements of $MX$ simply capacities.

There is considered in \cite{KZ} a tensor product for capacities, which is a continuous map $\otimes:MX_1\times MX_2\to M(X_1\times X_2)$ such that for each $i\in\{1,2\}$ we have $M(p_i)\circ\otimes= \pr_i$ where $p_i:X_1\times X_2\to X_i$ and $\pr_i:MX_1\times MX_2\to MX_i$ are natural projections. This definition is based on the capacity monad structure.  We give there a direct formulae for evaluating  tensor product of capacities. For $\mu_1\in MX_1$, $\mu_2\in MX_2$ and $B\subset X_1\times X_2$ we put $\mu_1\otimes\mu_2(B)=\sup\{t\in[0,1]\mid\mu_1(\{x\in X_1\mid \mu_2(p_2((\{x\}\times X_2)\cap B))\ge t\}\ge t\}$.  Note that, despite the space of capacities contains the space of probability measures,  the tensor product of capacities does not extend tensor product of probability measures. It was noticed in \cite{KZ} that we can extend the definition of tensor product to any finite number of factors by induction.

\begin{lemma}\label{P} Let $\A_i$ be a closed subset of a compactum $X_i$ and $\mu_i\in MX_i$ such that $\mu_i(X_i\setminus A_i)=0$ for each $i\in \{1,\dots,n\}$. Then $\otimes_{i=1}^n\mu_i(\prod_{i=1}^n X_i\setminus\prod_{i=1}^n A_i)=0$.
\end{lemma}

\begin{proof} Consider the case $n=2$.   Let $B$ any compact subset of $(X_1\times X_2)\setminus(A_1\times A_2)$. For any $t>0$ consider the set $K_t=\{x\in X_1\mid \mu_2(p_2((\{x\}\times X_2)\cap B))\ge t\}$. If $x\in A_1$, then $p_2((\{x\}\times X_2)\cap B)\subset X_2\setminus A_2$, hence $x\notin K_t$. Thus $K_t\subset X\setminus A_1$ and we obtain $\mu_1\otimes\mu_2(B)=0$.

The general case could be obtained by induction.
\end{proof}

Let us describe the Sugeno integral with respect to a capacity $\mu\in MX$. Fix any increasing homeomorphism $\psi:(0,1)\to\R$. We put additionally $\psi(0)=-\infty$, $\psi(1)=+\infty$ and assume $-\infty<t<+\infty$ for each $t\in\R$. We consider for each  function $f\in  C(X)$ an integral defined by the formulae
$$\int_X^{Sug} fd\mu=\max\{t\in\R\mid \mu(f^{-1}([t,+\infty)))\ge\psi^{-1}(t)\}$$

 Let us remark that we use some modification from \cite{R2} of Sugeno integral. The original Sugeno integral \cite{Su} "ignores" function values outside the interval $[0,1]$ and we introduce a "correction" homeomorphism $\psi$ to avoid this problem.

Now, we are going to introduce notion of  equilibrium under uncertainty for games where belief of each player about a choice of the strategy by the other player is a capacity. We follow definitions and denotation from \cite{EK} with the only difference that we use the Sugeno integral for expected payoff instead  the Choquet integral.

 We consider a $n$-players game $f:X=\prod_{i=1}^n X_i\to\R^n$ with compact Hausdorff spaces of strategies $X_i$. The coordinate function $p_i:X\to \R$ we call payoff function of $i$-th player. For $i\in\{1,\dots,n\}$ we denote by  $X_{-i}=\prod_{j\neq i} X_j$  the set of strategy combinations which players other
than $i$ could choose. For $x\in X$ the corresponding point in $X_{-i}$ we denote by $x_{-i}$.  In contrast to standard game theory, beliefs of $i$-th player about opponents'
behaviour are represented by non-additive measures (or capacities) on $X_{-i}$.

For $i\in\{1,\dots,n\}$  we consider the expected payoff function $P_i:X_i\times MX_{-i}\to\R$ defined as follows $P_i(x_i,\nu)=\int_{X_{-i}}^{Sug} p_{i}^{x_i}d\nu$ where the function $p_{i}^{x_i}:X_{-i}\to\R$ is defined by the formulae $p_{i}^{x_i}(x_{-i})=p_{i}(x_i,x_{-i})$, $x_i\in X_i$ and $\nu\in MX_{-i}$.

We are going to prove continuity of $P_i$. We will need some notations and a  technical lemma. Let $f:X\times Y\to\R$ be a  function. Consider any $x\in X$ and $t\in\R$. Denote $A^x_{\le t}=\{y\in Y\mid f(x,y)\le t\}$. We also will use analogous notations $A^x_{\ge t}$, $A^x_{<t}$ and $A^x_{>t}$.

\begin{lemma}\label{C0} Let $f:X\times Y\to\R$ be a continuous function on the product $X\times Y$ of compacta $X$ and $Y$. Then for each $x\in X$, $t\in\R$ and $\delta>0$ there exists an open neighborhood $O$ of $x$ such that $A^z_{\le t}\subset A^x_{<t+\delta}$ ($A^z_{\ge t}\subset A^x_{>t-\delta}$) for each $z\in O$.
\end{lemma}

\begin{proof} Let us prove the first statement of the lemma. We have two compact sets $\{x\}\times A^x_{\le t}$ and $\{x\}\times A^x_{\ge t+\delta}$. Consider its open neighborhoods $V=\{(z,y)\in X\times Y\mid f(z,y)<t+\frac{\delta}{2}\}$ and  $U=\{(z,y)\in X\times Y\mid f(z,y)>t+\frac{\delta}{2}\}$. Since $\{x\}\times A^x_{\le t}$ and $\{x\}\times A^x_{\ge t+\delta}$ are compact, we can choose an open set $O\subset X$ and two open sets $W_1$, $W_2\subset Y$ such that $\{x\}\times A^x_{\le t}\subset O\times W_1\subset V$ and  $\{x\}\times A^x_{\ge t+\delta}\subset O\times W_2\subset U$. Consider any $z\in O$ and $y\in Y$ such that $f(z,y)\le t$. Then $y\notin W_2\supset A^x_{\ge t+\delta}$, hence $y\in A^x_{<t+\delta}$.

The proof of the second statement is the same.
\end{proof}

\begin{lemma}\label{C} The map $P_i$ is continuous.
\end{lemma}

\begin{proof} Consider any $x\in X_i$ and $\nu_0\in MX_{-i}$ and put $P_i(x,\nu_0)=t\in\R$. Consider any $\varepsilon>0$. By Lemma \ref{C0} we can choose a neighborhood $O_1$ of $x$ such that for each $z\in O_1$ we have $A^z_{\ge t+\frac\varepsilon2}\subset A^x_{> t+\frac\varepsilon4}$. Put $V_1=\{\nu\in M(X_{-i}\mid \nu(A^x_{\ge t+\frac\varepsilon4})<\psi^{-1}(t+\frac\varepsilon4)\}$, then $V_1$ is a neighborhood of $\nu_0$.

We also can choose a neighborhood $O_2$ of $x$ such that for each $z\in O_2$ we have $A^z_{\le t-\frac\varepsilon2}\subset A^x_{<t-\frac\varepsilon4}$. Put $V_2=\{\nu\in M(X_{-i}\mid \nu(A^x_{\ge t-\frac\varepsilon4})>\psi^{-1}(t-\frac\varepsilon2)\}$, then $V_2$ is a neighborhood of $\nu_0$.

Put $O=O_1\cap O_2$ and $V=V_1\cap V_2$. Consider any $(z,\nu)\in O\times V$. Since $(z,\nu)\in O_1\times V_1$, we have $\nu(A^z_{\ge t+\frac\varepsilon2})\le\nu(A^x_{> t+\frac\varepsilon4})\le\nu(A^x_{\ge t+\frac\varepsilon4})<\psi^{-1}(t+\frac\varepsilon4)<\psi^{-1}(t+\frac\varepsilon2)$. Hence $P_i(z,\nu)<t+\varepsilon$.

On the other hand, since $(z,\nu)\in O_2\times V_2$, we have $\nu(A^z_{\ge t-\frac\varepsilon2})\ge\nu(A^z_{>t-\frac\varepsilon2})=\nu(X_{-i}\setminus A^z_{\le t-\frac\varepsilon2})\ge\nu(X_{-i}\setminus A^x_{<t-\frac\varepsilon4})=\nu(A^x_{\ge t-\frac\varepsilon4})>\psi^{-1}(t-\frac\varepsilon2)$. Hence $P_i(z,\nu)>t-\varepsilon$ and the map $P_i$ is continuous.
\end{proof}

For $\nu_i\in M(X_{-i})$ denote by $R_i=\{x\in X_i\mid P_i(x,\nu_i)=\max\{P_i(z,\nu_i)\mid z\in X_i\}$ the best response correspondence
of player $i$ given belief $\nu_i$. The set $R_i$ is well defined and compact by Lemma \ref{C}.

A belief system $(\nu_1,\dots,\nu_n)$, where $\nu_i\in M(X_{-i})$, is called {\it an equilibrium under uncertainty with Sugeno payoff} if for all $i$ we have $\nu_i(X_{-i}\setminus\prod_{j\ne i}R_j)=0$.

    The main goal of this paper is to prove the existence of such equilibrium. Since Sugeno integral does not preserve linear convexity on $MX$, we can not use methods from \cite{DW} and \cite{EK}. We will use some another natural convexity structure on the space of capacities which has the binarity property (has Helly number 2).

\section{Binary convexity on the space of capacities}

Consider a compactum $X$. There exists a natural lattice structure on $MX$ defined as follows $\nu\vee\mu(A)=\max\{\nu(A),\mu(A)\}$ and $\nu\wedge\mu(A)=\min\{\nu(A),\mu(A)\}$ for each closed subset $A\subset X$ and $\nu$, $\mu\in MX$. The lattice $MX$ is a compact complete sublattice of the lattice $[0,1]^\tau$ with natural coordinate-wise operations.
The lattice $MX$ has a greatest  element and a a least element defined as $\mu_{1X}(A)=1$ for each $A\neq \emptyset$, $\mu_{1X}(\emptyset)=0$ and $\mu_{0X}(A)=0$ for each $A\neq X$, $\mu_{0X}(X)=1$.

By convexity on $MX$ we mean any  family $\C$ of closed subsets which is stable for intersection
and contains $MX$ and the empty set. Elements of $\C$ are called
$\C$-convex (or simply convex). See \cite{vV} for more information about abstract convexities.

A convexity $\C$ on $MX$ is called $T_2$ if for each distinct $x_1$, $x_2\in
MX$ there exist $S_1$, $S_2\in\C$ such that $S_1\cup S_2=X$,
$x_1\notin S_2$ and $x_2\notin S_1$. Let $\L$ be a family of subsets of a compactum $X$. We say that  $\L$ is {\it linked} if the intersection of every  two elements is non-empty. A convexity $\C$ is called {\it binary} if the intersection of every  linked subsystem of $\C$ is non-empty.

For $\nu$, $\mu\in MX$ we denote $[\nu,\mu]=\{\alpha\in MX\mid \nu\wedge\mu\le\alpha\le\nu\vee\mu\}$. It is easy to see that $[\nu,\mu]$ is a closed subset of $MX$. We consider on $MX$ a convexity $\C_X=\{[\nu,\mu]\mid \nu, \mu\in MX\}$.

\begin{lemma}\label{Bin} The convexity $\C_X$ is binary.
\end{lemma}

\begin{proof} Let $\B$ is a linked subfamily of $\C$. It is enough to prove that intersection of every three elements of $\B$ is not empty by Proposition 2.1 from \cite{RC}. Consider any $[\mu_1,\nu_1]$, $[\mu_2,\nu_2]$, $[\mu_3,\nu_3]\in\B$. We can suppose that $\mu_i\le\nu_i$ for each $i\in\{1,2,3\}$. We denote by $\nu_A=\wedge\{\nu_i\mid i\in A\}$ and $\mu_A=\vee\{\mu_i\mid i\in A\}$ for each $A\subset\{1,2,3\}$. It is enough to prove that $\mu_{123}(B)\le\nu_{123}$.

Suppose the contrary then there  exists a closed set $B\subset X$ such that $\nu_{123}(B)<\mu_{123}(B)$.  We can choose $i\in\{1,2,3\}$ such that $\nu_{123}(B)<\mu_{i}(B)$.
Without loss of generality we can suppose that $i=1$. Since the family $\{[\mu_1,\nu_1]$, $[\mu_2,\nu_2]$, $[\mu_3,\nu_3]\}$ is linked, we have $\mu_1(B)\le\mu_{12}(B)\le\nu_{12}(B)$. Hence $\nu_3(B)<\mu_1(B)$. But then $\nu_{13}(B)\le\nu_3(B)<\mu_1(B)\le\mu_{13}(B)$ and we obtain a contradiction with the fact that $[\mu_1,\nu_1]\cap[\mu_3,\nu_3]=\emptyset$.
\end{proof}

\begin{lemma}\label{T2} The convexity $\C_X$ is $T_2$.
\end{lemma}

\begin{proof} Consider any $\mu$, $\nu\in MX$ such that $\mu\neq\nu$. Then there exists a closed subset $A\subset X$ such that $\mu(A)\neq\nu(A)$. We can suppose that $\mu(A)<\nu(A)$. Put $a=\frac{\mu(A)+\nu(A)}{2}$ and consider sets $O_1=\{\alpha\in MX\mid \alpha(A)\ge a\}$ and  $O_2=\{\alpha\in MX\mid \alpha(A)\le a\}$. Consider the capacity $\nu_1\in MX$ defined as follows $\nu_1(C)=0$ if $A\setminus C\neq\emptyset$,  $\nu_1(C)=a$ if $A\subset C$ and $C\neq X$ and $\nu_1(X)=1$ for a closed subset $C\subset X$. Then we have that $O_1=[\nu_1,\mu_{1X}]\in\C_X$.

Analogously, we can consider the capacity $\nu_2\in MX$ defined as follows $\nu_2(\emptyset)=0$  $\nu_2(C)=a$ if $\emptyset\neq C\subset A$  and $\nu_2(C)=1$ if $C\setminus A\neq\emptyset$ for a closed subset $C\subset X$. Then we have that $O_1=[\mu_{0X},\nu_2]\in\C_X$. Obviously, $O_1\cup O_2=MX$ and $\mu\notin O_1$ and $\nu\notin O_2$.
\end{proof}

\section{The main result}

We will  prove  existence of  equilibrium introduced in Section 2, moreover we will show that each $\nu_i$ could be represented as tensor product of capacities on factors.

By a multimap (set-valued map) of a set $X$ into a set $Y$ we mean a map $F:X\to 2^Y$. We use the notation $F:X\multimap Y$. If $X$ and $Y$ are topological spaces, then a multimap $F:X\multimap Y$ is called upper semi-continuous (USC) provided for each open set $O\subset Y$ the set $\{x\in X\mid F(x)\subset O\}$ is open in $X$. It is well-known that a multimap is USC iff its graph is closed in $X\times Y$.

Let  $F:X\multimap X$ be a  multimap. We say that a point $x\in X$ is a fixed point of $F$ if $x\in F(x)$.
The following counterpart of Kakutani theorem for binary convexity was obtained in \cite{RN}).

\begin{theorem}\label{A} Let $\C$ be a  $T_2$ binary convexity on a continuum $X$ and $F:X\multimap X$ is a USC multimap with values in $\C$. Then $F$ has a fixed point.
\end{theorem}

We use definitions and notations from Section 2.

\begin{theorem} There exists $(\mu_1,\dots,\mu_n)\in M(X_1)\times\dots\times M(X_n)$ such that $(\mu_1^*,\dots,\mu_n^*)$ is an   equilibrium under uncertainty with Sugeno payoff, where $\mu_i^*=\otimes_{j\neq i}\mu_j$
\end{theorem}

\begin{proof}  For each $i\in\{1,\dots,n\}$ consider a multimap $\gamma_i:\prod_{j=1}^nM(X_j)\multimap M(X_i)$ defined as follows $\gamma_i(\mu_1,\dots\mu_n)=\{\mu\in M(X_i)\mid \mu(X_i\setminus R_i(\mu_i^*))=0\}$. It follows from the definition of topology on $M(X_i)$ that $\gamma_i(\mu_1,\dots\mu_n)$ is a closed subset of $M(X_i)$ for each  $(\mu_1,\dots\mu_n)\in \prod_{j=1}^nM(X_j)$. Consider $\nu\in M(X_i)$ defined as follows $\nu(A)=1$ if $A\cap R(\mu_i^*)\neq\emptyset$ and $\nu(A)=0$ otherwise.
Then we have  $\gamma_i(\mu_1,\dots\mu_n)=[\mu_{X_i0},\nu]$, hence $\gamma_i(\mu_1,\dots\mu_n)\in\C_{X_i}$.

Define a multimap $\gamma:\prod_{j=1}^nM(X_j)\multimap \prod_{j=1}^nM(X_j)$ by the formulae $\gamma(\mu_1,\dots\mu_n)=\prod_{i=1}^n\gamma_i(\mu_1,\dots\mu_n)$. Let us show that $\gamma$ is USC. Consider any pair $(\mu,\nu)\in\prod_{j=1}^nM(X_j)\times\prod_{j=1}^nM(X_j)$ such that $\nu\notin \gamma(\mu)$. Then there exists $i\in\{1,\dots,n\}$ and a compactum $K\subset X_i\setminus R_i(\mu_i^*)$ such that $\nu_i(K)>0$. Put $O_\nu\{\alpha\in \prod_{j=1}^nM(X_j)\mid \alpha_i(K)>0\}$. Then $O_\nu$ is an open neighborhood of $\nu$. It follows from Lemma  \ref{C} and continuity of tensor product that there exists an open neighborhood $O_\mu$ of $\mu$ such that for each $\alpha\in O_\mu$ we have $R(\alpha_i^*)\cap K=\emptyset$. Hence for each $(\alpha,\beta)\in O_\mu\times O_\nu$ we have $\beta\notin\gamma(\alpha)$ and $\gamma$ is USC.

We consider on $\prod_{j=1}^nM(X_j)$ the family $\C=\{\prod_{i=1}^n C_i\mid C_i\in\C_{X_i}\}$. It is easy to see that $\C$ forms a $T_2$ binary convexity on a continuum $\prod_{j=1}^nM(X_j)$ (let us remark that each $M(X_j)$ is connected).
Then by Theorem \ref{A} $\gamma$ has a fixed point $\mu=(\mu_1,\dots\mu_n)\in \prod_{j=1}^nM(X_j)$. Let us show that $(\mu_1^*,\dots,\mu_n^*)$ is an   equilibrium under uncertainty. Consider any $i\in\{1,\dots,n\}$. Then $\mu_i(X_i\setminus R_i(\mu_i^*))=0$. We have by Lemma \ref{P} $\mu_i^*(\prod_{j\ne i}X_i\setminus\prod_{j\ne i}R_j(\mu_j^*))=0$.
\end{proof}

\begin{remark} Many results of our could be deduced from general results obtained in \cite{RN} but we give direct (not difficult) proofs here because otherwise it would require introducing additional categorical notions.
\end{remark}

\end{document}